\documentclass[review]{elsarticle}
\usepackage{lineno,hyperref}
\usepackage{amssymb,enumerate}
\usepackage{amsmath}
\usepackage{graphics}
\usepackage{color}
\usepackage{tikz}
\usetikzlibrary{positioning}
\usepackage{xcolor}

\usepackage{tikz}
\usetikzlibrary{calc,intersections,patterns}
\newtheorem{thm}{Theorem}[section]

\newtheorem{lem}[thm]{Lemma}

\newtheorem{defn}[thm]{Definition}
\newtheorem{rem}[thm]{Remark}

\newtheorem{proof}{Proof}[section]
\newcommand{\e}{\varepsilon}
\newcommand{\R}{\mathbf{R}}
\modulolinenumbers[5]
\journal{Journal of \LaTeX\ Templates}
\begin{document}

\begin{frontmatter}

\title{Blow-up of solutions to the Euler-Poisson-Darbox equation with
critical power nonlinearity}



\author[secondaryaddress]{Mengting Fan}
\ead{fanmengting@zjnu.edu.cn}

\author[secondaryaddress]{Ning-An Lai}
\ead{ninganlai@zjnu.edu.cn}

\author[thirdaddress]{Hiroyuki Takamura}
\ead{hiroyuki.takamura.a1@tohoku.ac.jp}

\address[secondaryaddress]{School of Mathematical Sciences, Zhejiang Normal University, Jinhua 321004, China}

\address[thirdaddress]{Mathematical Institute, Tohoku University, Aoba, Sendai 980-8578, Japan}

\begin{abstract}

In our recent precious work, we established the finite time blow up result and upper bound of lifespan estimate to the singular Cauchy problem of semilinear Euler-Poisson-Darboux equation in $\R^n$ with subcritical power type nonlinearity.
By introducing an improved test function, we obtain an enhanced lower bound for the functional including the spacetime integral of the nonlinear term (see \eqref{C14} below) with an additional logarithmic growth, which finally yields the blow up result and upper bound of lifespan estimate for the corresponding Cauchy problem with "critical" nonlinear power. And this gives some partial answer to the $\mathbf{open~ problem ~1}$ posed by D'Abbicco (J. Differential Equations 286 (2021), 531-556).

\end{abstract}


\end{frontmatter}


\section{Introduction}

\par\quad
In our recent work \cite{FLT}, the singular Cauchy problem of the semilinear Euler-Poisson-Darboux
equation
in $\R^n(n\ge 1)$
\begin{equation}
\label{SHEPD}
\left\{
\begin{aligned}
& u_{tt}-\Delta u+\frac{\mu}{t}u_t=t^\alpha|u|^p, ~~~~t>0, x\in \R^n,\\
& u(0, x)=u_0(x), \quad u_t(0, x)=0,~~~~x\in \R^n
\end{aligned}
\right.
\end{equation}
 was studied. Here $\mu>0$ denotes a parameter describing the effect of the damping, $\alpha\ge 0, p>1$ stand for nonlinear exponents. By using the modified Bessel function to construct a test function and combining the iteration argument, the finite time blow up result was established for the subcritical nonlinear power
\[
1<p<p_S(n+\mu, \alpha), n\ge 1, \mu>0,\alpha\geq0,
\]
where $p_S(n+\mu, \alpha)$ denotes the positive root of the quadratic equation
\begin{equation}\label{gmnmupalpha}
\begin{aligned}
\gamma(n, \mu, \alpha, p):=2+(n+\mu+1+2\alpha)p-(n+\mu-1)p^2=0.
\end{aligned}
\end{equation}
In this paper, we are going to study the \lq\lq critical" case, i.e. $p=p_S(n+\mu, \alpha)$. Here the quotation marks mean that criticality is tentative, since the global existence for the opposite side $p>p_S(n+\mu, \alpha)$ is left open for $n\ge 2$. And the global existence of weak solution in $\R$ (1-D) has been studied in \cite{Dab}, in which the author posed some \textbf{open problems in section 5} with a sentence like \lq\lq \textit {Also, a complete knowledge of blow-up results for the semilinear E.P.D. equation considered
in this paper is lacking so far}".

The study of Euler-Poisson-Darboux equation can be traced back to early time, and originated from some classical works such as Euler \cite{Euler70}, Poisson \cite{Pos23}, Riemann \cite{Rie76} and Darboux \cite{Dar15}. The E.P.D. equation (short for the Euler-Poisson-Darboux equation) has important application in physics, such as the
theory of surfaces, the sound propagation and the colliding gravitational fields, we refer the reader to \cite{Wei54, Wei55, Bre73, UKa13, WCh16} for more details.

Recently, the nonlinear E.P.D. equation and related models have attracted more and more attentions, see \cite{Kel57, Lev74, Lev76, Ues94, Zha18, DLu13, Dab} and detailed introduction in \cite{FLT}. These works mainly concern the global existence and finite time blow up under different assumptions. We should mention that the regular analog of problem \eqref{SHEPD} in $\R^n(n\ge 1)$
\begin{equation}
\label{REPD}
\left\{
\begin{aligned}
& u_{tt}-\Delta u+\frac{\mu}{t}u_t=t^\alpha|u|^p, ~~~~t\ge t_0>0, x\in \R^n,\\
& u(t_0, x)=u_0(x), \quad u_t(t_0, x)=u_1(x),~~~~x\in \R^n
\end{aligned}
\right.
\end{equation}
or
\begin{equation}
\label{Sdamped}
\left\{
\begin{aligned}
& u_{tt}-\Delta u+\frac{\mu}{1+t}u_t=t^\alpha|u|^p, ~~~~t\ge 0, x\in \R^n,\\
& u(0, x)=u_0(x), \quad u_t(0, x)=u_1(x),~~~~x\in \R^n
\end{aligned}
\right.
\end{equation}
has also been widely studied. Actually, for $\alpha=0$, it is conjectured that the small data Cauchy problem \eqref{Sdamped} admits a critical power
\begin{equation}\label{cp}
\begin{aligned}
p_{crit}=\max\Big\{p_S(n+\mu), 1+\frac 2n\Big\},~~~\mu>0,
\end{aligned}
\end{equation}
where $p_S(n)$ denotes the positive root of
\[
\gamma(n, p):=2+(n+1)p-(n-1)p^2=0,~~n\ge 2,
\]
 and is usually called the Strauss critical exponent, dividing the nonlinear power $p\in (1, \infty)$ for the small data Cauchy problem of semilinear wave equation
\[
u_{tt}-\Delta u=|u|^p
\]
into two parts: finite time blow up range $(1, p_S(n)]$ and global existence range $(p_S(n), \infty)$. The existence of such critical exponent was conjectured by Strauss \cite{Str81} in 1981 and has been completely solved.
The definition \eqref{cp} means $p_{crit}=p_S(n+\mu)$ if $0<\mu<\mu^*(n)$, and $p_{crit}=1+\frac 2n$ if $\mu\ge \mu^*(n)$, where
\begin{equation}\label{mun}
\mu^*(n)=\frac{n^2+n+2}{n+2}
\end{equation}
 is determined by setting
\[
p_S(n+\mu^*)=1+\frac 2n.
\]
This conjecture has been partially verified, see \cite{Dab15, Dab, DaL15, DLR15, HeL25, HLY24, HLY25, HSZ25, Ike, IKTW20, Wak16, KaS19, Lai20, LST20, LTW17, LZ21, LiY25, LiY26, Pal19, TuL19} and references therein.

As mentioned above, the aim of this paper is to establish blow up result and upper bound of lifespan estimate for the Cauchy problem \eqref{SHEPD} with nonlinear power $p=p_S(n+\mu, \alpha)$. It is expected that $p_S(n+\mu, \alpha)$ is really the critical power for $0<\mu<\mu^*(n)$, although this fact is still to be verified. Usually, the critical problem is much more delicate to get blow up result than the subcritical case, since it is the boardline between blow up and global existence, and the lifespan will be always the exponential type, meaning more close to the global existence. For this end, one usually has to gain an additional logarithmic growth for the integral of the nonlinear term, see the classical works of blow up by Yordanov and Zhang \cite{YZ06}, Zhou \cite{Zhou07}, and lifespan estimate by Takamura and Wakasa \cite{TW11} for semilinear wave equation. This observation is further developed in \cite{TuL19,PT22,LZ23,LPT24} for other wave models with critical nonlinear power.

In order to obtain the additional logarithmic growth for the critical problem, we have two different key ingredients comparing to the subcritical case. The first one is that we construct the test function by using rescaled functions and integrating with respect to the scale parameter (see \eqref{bq}), which leads to better asymptotic behavior (see \eqref{asymforbq} and \eqref{asymfortbq}). The other one is the construction of the blow up functional ($Y(M)$ in \eqref{C10}) which includes the spacetime integral of the nonlinear term and admits an enhanced lower bound (see \eqref{C13}).

Before stating the main results, we first denote the energy and weak solutions of problem \eqref{SHEPD}.
\begin{defn}
We say that $u$ is an energy solution of \eqref{SHEPD} with Cauchy data $u_0(x)\in H^1(\R^n)$ over $[0,T)$ if
\begin{equation}
\begin{aligned}
u\in C([0,T),H^1(\R^n))\cap C^1([0,T),L^2(\R^n))\cap L^p_{loc}((0, T)\times\R^n)
\end{aligned}
\end{equation}
satisfies
\begin{equation}
\begin{aligned}
&\int_{0}^{t}ds\int_{\R^n}\{-u_t(s, x)\phi_t(x,s)+\nabla u(s, x)\cdot \nabla \phi(s, x)\}dx\\
&+\int_{0}^{t}ds\int_{\R^n} \frac{\mu u_t(s, x)}{s}\phi(s, x)dx\\
=&\int_{0}^{t}ds\int_{\R^n}s^\alpha|u(s, x)|^p\phi(s, x)dx
\end{aligned}
\end{equation}
with any $\phi(t, x)\in C_0^\infty([0, T)\times\R^n)$.

Employing the integration by parts in the above equality and letting $t\rightarrow T$, we have that
\begin{equation}
\begin{aligned}
&\int_{[0, T)\times\R^n}u(s, x)\bigg\{\phi_{tt}(s, x)-\Delta \phi(s, x)-\Big(\frac{\mu\phi(s, x)}{s}\Big)_s\bigg\}dxds\\
&+\int_{0}^{t}ds\int_{\R^n} \frac{\mu u_t(s, x)}{s}\phi(s, x)dx\\
=&-\int_{\R^n}u_0(x)\phi_t(0, x)dx+\int_{0}^{t}ds\int_{\R^n}s^\alpha|u(s, x)|^p\phi(s, x)dx,
\end{aligned}
\end{equation}
which is exactly the definition of the weak solution to \eqref{SHEPD}.
\end{defn}

The main result is stated as
\begin{thm}
\label{hddamped}
Let
\begin{equation}
\label{al}
\left\{
\begin{aligned}
& \alpha\ge 0,~~~n\ge 3,\\
& \alpha>0,~~~~n=2
\end{aligned}
\right.
\end{equation}
and
\[
0<\mu\leq\mu^*(n,\alpha), p=p_S(n+\mu, \alpha),
\]
 where
\begin{equation}
\begin{aligned}
\mu^*(n,\alpha)=\frac{2n^2+(n+2+\alpha)(n\alpha+2+\alpha)}{(n+2+\alpha)(2+\alpha)}.
\end{aligned}
\end{equation}
Assume the initial data $u_0(x)\in H^1(\R^n)$ is positive and has compact support
\begin{equation}
\begin{aligned}
supp~u_0(x)\subset\{|x|\le 1\},
\end{aligned}
\end{equation}
then there exists a constant $\e_0=\e_0(u_0,n,p,\mu,\alpha)>0$ such that the lifespan of the energy solution to the Cauchy problem \eqref{SHEPD} satisfies
\begin{equation}
\label{lf2}
T(\e)\le \exp\left(C\e^{-p(p-1)}\right),
\end{equation}
where $C$ denotes a generic positive constant independent of $\e$ which may have different values from line to line.
\end{thm}
\begin{rem}
The value of $\mu^*(n, \alpha)$ in the above theorem comes from the fact that
 $p_{S}(n,\alpha,\mu)\geq p_F(n,\alpha)$, where  $p_F(n,\alpha)=1+\frac{2+\alpha}{n}$
and $p_{S}(n,\mu,\alpha)$ is the positive root of the quadratic equation \eqref{gmnmupalpha}, denoting the shifted Fujita and Strauss exponent by $\alpha$ and $\mu$ respectively. It can be obtained by solving the inequality
 $\gamma(n,\mu,\alpha,1+\frac{2+\alpha}{n})\geq0$. It is easy to see that if  $\alpha=0$, then $\mu^*(n, \alpha)$ is nothing but $\mu^*(n)$ in \eqref{mun}.

\end{rem}

\begin{rem}

Comparing to the blow up result for the corresponding regular Cauchy problem \eqref{REPD} or \eqref{Sdamped}, we face a difficulty to handle the singular damping coefficient at $t=0$ after multiplying the test function and integrating by part, see the first term in \eqref{C17} below. To overcome this difficulty, we use the modified Bessel function $K_{\frac{\mu-1}{2}}(t)$ to construct an appropriate test function (see \eqref{Kmu}). The key point is that we can use the asymptotic behavior $\eqref{Bes}_1$ and recurrence relation \eqref{rela} of $K_{\frac{\mu-1}{2}}(t)$ to eliminate the singularity at $t=0$, see \eqref{C00} below.

\end{rem}

\section{Local existence and finite speed propagation}
In this section we will prove the local existence of solution and finite speed propagation property to the Cauchy problem \eqref{SHEPD}, by using the method similar to that in Section 5 in \cite{LST}. Assuming that
\begin{equation}
\label{p}
\left\{
\begin{aligned}
& 1<p\leq\frac{n}{n-2},~~~n\ge 3,\\
& 1<p<\infty,~~~~n=2.
\end{aligned}
\right.
\end{equation}

Define the function space
\begin{equation}\label{btk}
\begin{aligned}
B_{T}:=&\Big\{u\in C\big([0,T),H^1(\R^n)\big)\cap C^1\big([0,T),L^2(\R^n)\big):\\
& supp~u\subset \big\{(t, x)\in[0, T)\times\R^n:|x|\leq t+1\big\},\|u\|_{B_{T}}<\infty\},
\end{aligned}
\end{equation}
where $T$ is a positive constant and
\begin{equation}\label{phibtk}
\begin{aligned}
\|u\|_{B_{T}}:= \mathop{sup}\limits_{t\in[0,T)}E_u^{\frac{1}{2}}(t),~~~E_u(t):=\frac{1}{2}\int_{\R^n}(u_t^2+|\nabla u|^2)dx.
\end{aligned}
\end{equation}
It is easy to see that $(B_{T},\|\cdot\|_{B_T})$ is a Banach space.

Consider the following Cauchy problem for $v\in B_{T}$
\begin{equation}
\label{repdfv}
\left\{
\begin{aligned}
& u_{tt}-\Delta u+\frac{\mu}{t}u_t=t^\alpha|v|^p=:F_v(t, x), ~~~~~(0, T)\times\R^n,\\
& u(0, x)=u_0(x), \quad u_t(0, x)=0,~~~~x\in \R^n.
\end{aligned}
\right.
\end{equation}
For fixed $T>0$, noting the assumption \eqref{p}, by Sobolev embedding inequality it holds that
\begin{equation}
\begin{aligned}
F_v(t, x)\in L^2\big((0, T)\times\R^n\big).
\end{aligned}
\end{equation}
We are going to show that the map
\begin{equation}\label{mv}
\begin{aligned}
M:v\mapsto u=Mv,~v\in B_{T}
\end{aligned}
\end{equation}
is a contracting mapping. For $v\in B_{T}$, by Gagliardo-Nirenberg inequality and Poincar\'{e} type inequality, it holds
\begin{equation}\label{vl2ptheta}
\begin{aligned}
||v||_{L^{2p}(\R^n)}\leq C||v||_{L^2{(\R^n)}}^{1-\theta}||\nabla v||_{L^2{(\R^n)}}^{\theta},~~\theta:=n\left(\frac{1}{2}-\frac{1}{2p}\right)
\end{aligned}
\end{equation}
and
\[
||v||_{L^{2}(\R^n)}\leq C(1+t)||\nabla v||_{L^2{(\R^n)}},
\]
plugging which into \eqref{vl2ptheta}, we get
\begin{equation}\label{vl2p}
\begin{aligned}
||v||_{L^{2p}(\R^n)}\leq C(1+t)^{1-\theta}||\nabla v||_{L^2{(\R^n)}}\leq C(1+t)^{1-\theta}E_v^{\frac{1}{2}}.
\end{aligned}
\end{equation}

We first show the finite speed propagation property of the energy solution. Without loss of generality, we may assume that $u(t, x)\in C^1([0, T); C_0^\infty(\R^n))$. In the general case, we may approximate $u_0(x)\in H^1(\R^n)$ and $F_v(t, x)\in L^2\big([0, T)\times\R^n\big)$ by the sequence $\{u_0^k\}\subset C_0^\infty(\R^n)$ and $\{F_v^k\}\subset C^1\left([0, T); C_0^\infty(\R^n)\right)$, and denote by $u^k$ the corresponding solutions to \eqref{repdfv}. For $(t, x)\in (0, T]\times \R^n$ and $\tau\in [0, t)$, define the truncated cone
\begin{equation}\label{cx0t0}
\begin{aligned}
C_{\tau}:=\Big\{(s, y): |y-x|\le t-s, 0\le s\leq\tau\Big\}.
\end{aligned}
\end{equation}
Multiplying the equation in \eqref{repdfv} with $u_t$, we have
\begin{equation}\label{M1}
\begin{aligned}
\partial_t\left(\frac12u_t^2+\frac12|\nabla u|^2\right)-\nabla\cdot(\nabla u\nabla u_t)+\frac{\mu}{t}u_t^2=F_vu_t,
\end{aligned}
\end{equation}
integrating which over $C_\tau$ yields
\begin{equation}\label{M2}
\begin{aligned}
&\frac12\int_{|x-y|\le t-\tau}\left(u_\tau^2+|\nabla u|^2\right)(\tau, y)dy-\frac12\int_{|x-y|\le t}\left(u_t^2+|\nabla u|^2\right)(0, y)dy\\
&+\frac1{\sqrt 2}\int_0^\tau\int_{|y-x|=t-\tau}\left[\frac12\left(u_s^2+|\nabla u|^2\right)+\frac{y-x}{|y-x|}\cdot\nabla uu_s\right]d\sigma ds\\
&+\int_0^\tau\int_{|y-x|\le t-\tau}\frac{\mu}{s}u_s^2 dyds\\
=&\int_0^\tau\int_{|y-x|\le t-s}F_vu_s dy ds.\\
\end{aligned}
\end{equation}
It is easy to see the third surface integral term in the left hand side is nonnegative by using Cauchy-Schwarz inequality
\begin{equation}\label{M3}
\begin{aligned}
\left|\frac{y-x}{|y-x|}\cdot\nabla uu_t\right|\le \frac12\left(u_t^2+|\nabla u|^2\right).
\end{aligned}
\end{equation}
Denoting the local energy by
\[
E_l(\tau)=\frac12\int_{|x-y|\le t-\tau}\left(u_\tau^2+|\nabla u|^2\right)(\tau, y)dy,
\]
then it follows from \eqref{M2} that
\begin{equation}\label{M4}
\begin{aligned}
E_l(\tau)\le E_l(0)+C\max_{s\in [0, \tau]}\|F_v\|_{L^2(\R^n)}\int_0^\tau E_l(s)ds,
\end{aligned}
\end{equation}
which yields by Gronwall's inequality
\begin{equation}\label{M5}
\begin{aligned}
E_l(\tau)\le E_l(0)\exp\left(C\max_{s\in [0, \tau]}\|F_v\|_{L^2(\R^n)}\tau\right).
\end{aligned}
\end{equation}
Hence if $u(0, y)=0$ for $|y-x|\le t$, then $u(\tau, y)=0$ for $|y-x|\le t-\tau, \tau\in [0, t)$. This implies if $u(0, x)=0$ for $|x|>1$, then $u(t, x)=0$ for $|x|\ge t+1$.

Next, we show that $||Mv||_{B_{T}}<\infty$.
By \eqref{vl2p} it holds
\begin{equation}\label{vput}
\begin{aligned}
\int_{\R^n}|v|^p|u_t|dx&\leq\Bigg(\int_{\R^n}|v|^{2p}dx\Bigg)^{\frac{1}{2}}\sqrt{2}E_u^{\frac{1}{2}}(t)\\
&\leq C(t+1)^{p(1-\theta)}E_v^{\frac{p}{2}}(t)E_u^{\frac{1}{2}}(t).
\end{aligned}
\end{equation}
Integrating \eqref{M1} over $[0, t]\times\R^n, t\in (0, T]$ and using the divergence theorem, we obtain
\begin{equation}\label{divesti}
\begin{aligned}
&E_u(t)+\int_{0}^{t}\int_{\R^n}\frac{\mu}{s}u_s^2dxds\\
\leq& E_u(0)+C\int_{0}^ts^\alpha (s+1)^{p(1-\theta)}E_u^{\frac{1}{2}}(s)ds,
\end{aligned}
\end{equation}
which yields further by Bihari's inequality
\begin{equation}\label{divesti12}
\begin{aligned}
E_u^{\frac{1}{2}}(t)&\leq E_u^{\frac{1}{2}}(0)+C\int_{0}^ts^\alpha (s+1)^{p(1-\theta)}ds.\\
&\leq E_u^{\frac{1}{2}}(0)+C(1+T)^{p(1-\theta)}T^{\alpha+1},\\
\end{aligned}
\end{equation}
hence for a fixed $T$ we have $E_u^{\frac{1}{2}}(t)<\infty$.

Finally, we show the contraction of the map $M$. Let $v_1,v_2\in B_{T}$ and
\begin{equation}\label{u1u2}
\begin{aligned}
u_1=Mv_1,~~~u_2=Mv_2
\end{aligned}
\end{equation}
with the same initial data and
\[
\overline{u}=u_1-u_2,~~\overline{v}=v_1-v_2.
\]
Then $\overline{u}$ satisfies the following Cauchy problem
\begin{equation}
\label{repdoverline}
\left\{
\begin{aligned}
& \overline{u}_{tt}-\Delta \overline{u}+\frac{\mu}{t}\overline{u}_t=t^\alpha|v_1|^p-t^\alpha|v_2|^p, ~~~~in~ (0, T)\times\R^n,\\
& \overline{u}(0, x)=0, \quad \overline{u}_t(0, x)\equiv0,~~~~x\in \R^n,
\end{aligned}
\right.
\end{equation}
and
\begin{equation}\label{divoverline}
\begin{aligned}
\frac{\partial}{\partial t}\Bigg(\frac{1}{2}\big(\overline{u}_t^2+|\nabla \overline{u}|^2\big)\Bigg)+\frac{\mu}{t}\left(\overline{u}_t\right)^2=\nabla\cdot
(\overline{u}_t\nabla \overline{u})+t^\alpha(|v_1|^p-|v_2|^p)\overline{u}_t.
\end{aligned}
\end{equation}
By \eqref{vl2p} and H\"{o}lder inequality, it holds
\begin{equation}
\begin{aligned}
&\int_{\R^n}t^\alpha\big||v_1|^p-|v_2|^p\big||\overline{u}_t|dx\\
\leq &C\int_{\R^n}t^\alpha|v_1-v_2|(|v_1|+|v_2|)^{p-1}|\overline{u}_t|dx\\
\leq &Ct^\alpha\left\||\overline{v}|(|v_1|+|v_2|)^{p-1}\right\|_{L^2(\R^n)}
\|\overline{u}_t\|_{L^2(\R^n)}\\
\leq& Ct^\alpha\|\overline{v}\|_{L^{2p}(\R^n)}\Big(\|v_1\|_{L_{(\R^n)}^{2p}}^{p-1}+
\|v_2\|_{L^{2p}(\R^n)}^{p-1}\Big)\|\overline{u}_t\|_{L^2(\R^n)}\\
\leq &C t^\alpha(t+1)^{p(1-\theta)}E_{\overline{v}}^{\frac{1}{2}}E_{\overline{u}}^{\frac{1}{2}}.
\end{aligned}
\end{equation}
Then, integrating \eqref{divoverline} over $[0, T]\times\R^n$, exploiting again the Bihari's inequality and proceeding similarly as above, we obtain
\begin{equation}
\begin{aligned}
||\overline{u}||_{B_{T}}\leq C(1+T)^{p(1-\theta)}T^{1+\alpha}\|\overline{v}\|_{B_{T}},
\end{aligned}
\end{equation}
which implies $M$ is a contraction map by choosing $T$ small enough.
The proof of the local existence and finite speed propagation are completed.

\section{Proof of Theorem \ref{hddamped}}
\subsection{Preliminary}
We first introduce a smooth cut-off function
\[
\eta(t):=
 \left\{
 \begin{array}{cl}
 1 & \mbox{for}\ t\le\frac12,\\
 \mbox{decreasing} & \mbox{for}\ \frac12<t<1,\\
 0 & \mbox{for}\ t\ge 1,\\
 \end{array}
 \right.
\]
and the corresponding scaled one is defined by
\begin{equation}
\label{scf}
 \eta_t(s):=\eta\left(\frac st\right),
\end{equation}
with $t>1$.

Besides, the solution for the following second order ordinary differential equation \begin{equation}
\label{ODE}
h''(t)-\left(\frac{\mu}{t}h(t)\right)_t-h(t)=0
\end{equation}
will be used.
\begin{lem}\label{ODE1}[Lemma 2.1 in \cite{FLT}]
The second order ordinary differential equation \eqref{ODE} admits a solution $h(t)$ satisfying
\begin{equation}\label{beode}
\begin{aligned}
& \lim_{t\rightarrow 0}\left(-h'(t)+\mu\frac{h(t)}{t}\right)\thicksim 2^{\frac{\mu-1}{2}}\Gamma\left(\frac{\mu+1}{2}\right)\triangleq C_0>0,\\
&h(t)\thicksim t^{\frac{\mu}2}e^{-t}\quad\mbox{for}\ t>R,\\
&|h'(t)|\thicksim t^{\frac{\mu}2}e^{-t}\quad\mbox{for}\ t>R,\\
\end{aligned}
\end{equation}
where $\Gamma$ denotes the gamma function, $R$ denotes a positive constant large enough.
\end{lem}
Actually, we may set
\begin{equation}\label{Kmu}
h(t):=t^{\frac{\mu+1}{2}}K_{\frac{\mu-1}{2}}(t),
\end{equation}
where $K_{\frac{\mu-1}{2}}(t)$ denotes the modified Bessel functions satisfying
\begin{equation}
\label{bessel}
K''_{\frac{\mu-1}{2}}(t)+\frac1tK'_{\frac{\mu-1}{2}}(t)-\left(1+\frac{\left(
\frac{\mu-1}2\right)^2}{t^2}\right)K_{\frac{\mu-1}{2}}(t)=0
\end{equation}
and
\begin{equation}\label{Bes}
\begin{aligned}
&K_\nu(z)\thicksim \frac12\Gamma(\nu)\left(\frac z2\right)^{-\nu},~~~\nu>0, z\rightarrow 0,\\
&K_{\nu}(z)=K_{-\nu}(z),\\
&K_\nu(z)\thicksim \frac{1}{\sqrt{2\pi}}z^{-\frac12}e^{-z},~~~|\mbox{arg}(z)|<\frac{\pi}2,\forall \nu, z\rightarrow \infty\\
\end{aligned}
\end{equation}
and the recurrence relations
\begin{equation}\label{rela}
\begin{aligned}
&K_\nu'(z)=-\frac12\left[K_{\nu-1}(z)+K_{\nu+1}(z)\right],\\
&K_\nu(z)=-\frac{z}{2\nu}\left[K_{\nu-1}(z)-K_{\nu+1}(z)\right].\\
\end{aligned}
\end{equation}
One can find more details about the property of the modified Bessel functions in \cite{Abr}.

Let
\begin{equation}\label{def:phi}
	\phi(x):=\int_{\mathbb{S}^{n-1}} e^{x\cdot\omega} \text{d}\omega,~~~ n \ge 2,
\end{equation}
which satisfies
\begin{equation}\label{testpro}
\begin{aligned}
0<\phi(x)\le C(1+|x|)^{-\frac{n-1}{2}}e^{|x|}.
\end{aligned}
\end{equation}

Denote
\begin{equation}\label{C1}
\begin{aligned}
h_\lambda(t):=h(\lambda t),~~~\phi_\lambda (x):=\phi(\lambda x),
\end{aligned}
\end{equation}
and
let
\begin{equation}\label{bq}
\begin{aligned}
b_q(t,x)=\int_{0}^{1}h_\lambda(t)\phi_{\lambda}(x)\lambda^{q-1}d\lambda,
\end{aligned}
\end{equation}
with
\begin{equation}\label{q}
\begin{aligned}
q=\frac{n-\mu-1}{2}-\frac{1}{p}=n+\alpha-\frac{n+\mu-1}{2}p,
\end{aligned}
\end{equation}
since $p=p_S(n+\mu, \alpha)$.
\begin{lem}\label{lbq}
The function $b_q(t, x)$ defined in \eqref{bq} satisfies
\begin{equation}\label{eqforbq}
\partial_t^2b_q-\Delta b_q-\partial_t\left(\frac{\mu}{t}b_q\right)=0,~~~t>0, x\in\R^n,
\end{equation}
\begin{equation}\label{asymforbq}
\begin{aligned}
&b_q(t, x)\thicksim t^{-q},~~~t\ge 1,
\end{aligned}
\end{equation}
and
\begin{equation}\label{asymfortbq}
\begin{aligned}
&\partial_t b_q(t, x)\lesssim t^{\frac{\mu}{2}-\frac{n-1}{2}}(t+2-r)^{\frac{1}{p}-1},~~~t\ge 1.
\end{aligned}
\end{equation}
\end{lem}
\begin{proof}
We may compute $h'(\lambda t)$ by using \eqref{rela}
\begin{equation}\label{dh}
\begin{aligned}
&h'(\lambda t)\\
=&\frac{\mu+1}{2}(\lambda t)^{\frac{\mu-1}{2}}\lambda K_{\frac{\mu-1}{2}}(\lambda t)
+(\lambda t)^{\frac{\mu+1}{2}}\lambda K'_{\frac{\mu-1}{2}}(\lambda t)\\
=&\frac{\mu+1}{2}(\lambda t)^{\frac{\mu-1}{2}}\lambda K_{\frac{\mu-1}{2}}(\lambda t)+(\lambda t)^{\frac{\mu+1}{2}}\lambda
\left(-\frac12K_{\frac{\mu-3}{2}}(\lambda t)-\frac12K_{\frac{\mu+1}{2}}(\lambda t)\right)\\
=&\frac{\mu+1}{2}(\lambda t)^{\frac{\mu-1}{2}}\lambda K_{\frac{\mu-1}{2}}(\lambda t)-\frac12(\lambda t)^{\frac{\mu+1}{2}}\lambda
\left(K_{\frac{\mu+1}{2}}(\lambda t)-\frac{2\cdot\frac{\mu-1}{2}}{\lambda t}
K_{\frac{\mu-1}{2}}(\lambda t)\right)\\
&-\frac12(\lambda t)^{\frac{\mu+1}{2}}\lambda K_{\frac{\mu+1}{2}}(\lambda t)\\
=&\mu (\lambda t)^{\frac{\mu-1}{2}}\lambda K_{\frac{\mu-1}{2}}(\lambda t)-(\lambda t)^{\frac{\mu+1}{2}}\lambda K_{\frac{\mu+1}{2}}(\lambda t).\\
\end{aligned}
\end{equation}
From the definition of $h_\lambda(t)$ and $b_q(t, x)$, direct computation yields
\begin{equation}\label{l41}
\begin{aligned}
\partial_tb_q(t,x)&=\int_{0}^{1}h_\lambda'(t)\phi_\lambda(x)\lambda^{q-1}d\lambda\\
&=\frac{\mu}{t}\int_{0}^{1}(\lambda t)^{\frac{\mu+1}{2}}K_{\frac{\mu-1}{2}}(\lambda t)\phi_\lambda(x)\lambda^{q-1}d\lambda\\
&-\int_{0}^{1}(\lambda t)^{\frac{\mu+1}{2}}K_{\frac{\mu+1}{2}}(\lambda t)\phi_\lambda(x)\lambda^{q}d\lambda.
\end{aligned}
\end{equation}
In a similar way, we have
\begin{equation}\label{l42}
\begin{aligned}
\partial_t^2b_q=&-\frac{\mu}{t^2}b_q+\frac{\mu}t\int_0^1\partial_th_\lambda(t)
\phi_\lambda(x)\lambda^{q-1}d\lambda\\
&-\frac{\mu+1}{2}\int_0^1(\lambda t)^{\frac{\mu-1}{2}}K_{\frac{\mu+1}{2}}(\lambda t)\phi_\lambda(x)\lambda^{q+1}d\lambda\\
&-\int_0^1(\lambda t)^{\frac{\mu+1}{2}}\partial_tK_{\frac{\mu+1}{2}}(\lambda t)\phi_\lambda(x)\lambda^{q}d\lambda\\
=&-\frac{\mu}{t^2}b_q+\frac{\mu^2}{t^2}b_q-\frac{\mu}t\int_0^1(\lambda t)^{\frac{\mu+1}{2}}K_{\frac{\mu+1}{2}}(\lambda t)\phi_\lambda(x)\lambda^{q}d\lambda\\
&+\int_0^1(\lambda t)^{\frac{\mu+1}{2}}K_{\frac{\mu-1}{2}}(\lambda t)\phi_\lambda(x)\lambda^{q+1}d\lambda.\\
\end{aligned}
\end{equation}
Noting that
\begin{equation}\label{l42a}
\begin{aligned}
\Delta b_q=\int_0^1(\lambda t)^{\frac{\mu+1}{2}}K_{\frac{\mu-1}{2}}(\lambda t)\phi_\lambda(x)\lambda^{q+1}d\lambda,
\end{aligned}
\end{equation}
then \eqref{eqforbq} follows by combining \eqref{l41}, \eqref{l42} and \eqref{l42a}.

We next show the asymptotic behavior \eqref{asymforbq}. For $t\ge 1$, by $\eqref{beode}_2$, one has
\begin{equation}\label{l43}
\begin{aligned}
b_q(t, x)\gtrsim& \int_{\frac{1}{2(t+1)}}^{\frac{1}{t+1}}(\lambda t)^{\frac{\mu}{2}}
e^{-\lambda t}\lambda^{q-1}d\lambda\\
\gtrsim& t^{\frac{\mu}2}\int_{\frac{1}{2(t+1)}}^{\frac{1}{t+1}}\lambda^{q-1+\frac{\mu}{2}}
e^{-\lambda (t+R)}d\lambda\\
\gtrsim& t^{-q}\int_{\frac12}^1s^{q+\frac {\mu}{2}-1}e^{-s}ds\\
\gtrsim& t^{-q}.\\
\end{aligned}
\end{equation}
In the opposite direction, we divide the proof into two parts. For $r\le \frac{t+1}{2}$, it is easy to get by $\eqref{beode}_2$ and \eqref{testpro}
\begin{equation}\label{l44}
\begin{aligned}
b_q(t, x)\lesssim& \int_0^1(\lambda t)^{\frac{\mu}2}e^{-\lambda t}(1+\lambda r)^{-\frac{n-1}2}e^{\lambda r}\lambda^{q-1}d\lambda\\
\lesssim&t^{\frac{\mu}2}\int_0^1e^{-\frac{\lambda(t+1)}{2}}
\lambda^{q+\frac{\mu}2-1}d\lambda\\
\lesssim&t^{-q}\int_0^{\infty}e^{-s}s^{q+\frac{\mu}{2}-1}ds\\
\lesssim&t^{-q},
\end{aligned}
\end{equation}
while for $\frac{t+1}{2}\le r\le t+1$, it holds that
\begin{equation}\label{bigr}
\begin{aligned}
b_q(t, x)\lesssim& \int_0^1(\lambda t)^{\frac{\mu}2}e^{-\lambda t}(1+\lambda r)^{-\frac{n-1}2}e^{\lambda r}\lambda^{q-1}d\lambda\\
\lesssim&\int_0^1(\lambda t)^{\frac \mu 2}\bigg(1+\lambda(t+1)\bigg)^{-\frac{n-1}{2}}\lambda^{q-1} d\lambda\\
\lesssim&t^{-q}\int_0^{\infty}(1+s)^{-\frac{n-1}{2}}s^{\frac \mu 2+q-1}ds\\
\lesssim&t^{-q}\left(\int_0^{1}(1+s)^{-\frac{n-1}{2}}s^{\frac{n-1}2-\frac 1p-1}ds+\int_1^{\infty}(1+s)^{-\frac 1p-1}ds\right),\\
\end{aligned}
\end{equation}
where we use the fact that
\[
q=\frac{n-\mu-1}{2}-\frac{1}{p}.
\]
Obviously, the last term in \eqref{bigr} in integrable, while for the second to the last term, we should require $\frac{n-1}{2}>\frac 1p$ to ensure the integrability.
Actually it is equivalent to $p>\frac 2{n-1}, n\ge 2$, which is also equivalent to
\begin{equation}\label{eq1}
\begin{aligned}
\gamma\left(n, \mu, \alpha, \frac 2{n-1}\right)=4(n-1)^2+2(n-3)\mu+4(n-1)\alpha
>0,
\end{aligned}
\end{equation}
which always holds for $n\ge 3, \alpha\ge 0, \mu>0$. If $n=2$, \eqref{eq1} becomes to
\begin{equation}\label{eq2}
\begin{aligned}
\mu<2+2\alpha,
\end{aligned}
\end{equation}
which also holds for $0<\mu\le \mu^*(2, \alpha)$ and $\alpha>0$. And hence under the assumption of Theorem \ref{hddamped}, the inequality \eqref{bigr} becomes
\begin{equation}\label{bigr6}
\begin{aligned}
b_q(t, x)\lesssim t^{-q}.
\end{aligned}
\end{equation}
Then \eqref{asymforbq} follows by combining \eqref{l43}, \eqref{l44} and \eqref{bigr6}.

We finally show the asymptotic behavior \eqref{asymfortbq}. We can write \eqref{l41} as
\begin{equation}\label{l45}
\begin{aligned}
\partial_tb_q(t,x)&=I+II,
\end{aligned}
\end{equation}
with
\[
\begin{aligned}
I&:=\frac{\mu}{t}\int_{0}^{1}(\lambda t)^{\frac{\mu+1}{2}}K_{\frac{\mu-1}{2}}(\lambda t)\phi_\lambda(x)\lambda^{q-1}d\lambda,\\
II&:=-\int_{0}^{1}(\lambda t)^{\frac{\mu+1}{2}}K_{\frac{\mu+1}{2}}(\lambda t)\phi_\lambda(x)\lambda^{q}d\lambda.
\end{aligned}
\]
By combining $\eqref{beode}_2, \eqref{beode}_3$, \eqref{testpro}, it is easy to estimate $I$ as
\begin{equation}
\begin{aligned}
|I|\leq \frac{\mu}{t}\int_{0}^{1}h_\lambda(t)\phi_\lambda(x)\lambda^{q-1}d\lambda\lesssim t^{-q-1}.
\end{aligned}
\end{equation}
For $II$, we split the proof into two parts: $0<r\leq\frac{t+1}{2}$ and $\frac{t+1}{2}\leq r \leq t+1$. For the former case, we have
\begin{equation}
\begin{aligned}
|II|&\lesssim \int_{0}^{1}(\lambda t)^{\frac{\mu+1}{2}}(\lambda t)^{-\frac{1}{2}}e^{-\lambda t}(1+\lambda r)^{-\frac{n-1}{2}}e^{\lambda r}\lambda^{q}d\lambda\\
&\lesssim \int_{0}^{1}(\lambda t)^{\frac{\mu}{2}}e^{-\lambda(t-r)}(1+\lambda r)^{-\frac{n-1}{2}}\lambda^{q}d\lambda\\
&\lesssim t^{\frac{\mu}{2}}\int_{0}^{1}e^{-\lambda(t-r)}(1+\lambda r)^{-\frac{n-1}{2}}\lambda^{q+\frac{\mu}{2}}d\lambda\\
&\lesssim t^{\frac{\mu}{2}}\int_{0}^{1}e^{-\frac{\lambda(t+1)}{2}}\lambda^{q+\frac{\mu}{2}}d\lambda\\
&\lesssim t^{-q-1}\int_{0}^{\infty}e^{-s}s^{\frac{n-1}{2}-\frac{1}{p}}ds\\
&\lesssim t^{-\frac{n-\mu-1}{2}+\frac{1}{p}-1},
\end{aligned}
\end{equation}
where
\[
q=\frac{n-\mu-1}{2}-\frac{1}{p}.
\]
If $\frac{t+1}{2}\leq r \leq t+1$, we get
\begin{equation}
\begin{aligned}
|II|&\lesssim t^{\frac{\mu}{2}}\int_{0}^{1}e^{-\lambda(t+2-r)}\big(\lambda (t+1)\big)^{-\frac{n-1}{2}}\lambda^{q+\frac{\mu}{2}}d\lambda\\
&\lesssim t^{\frac{\mu}{2}}\int_{0}^{t+2-r}e^{-s}(t+1)^{-\frac{n-1}{2}}s^{-\frac{n-1}{2}}(t+2-r)^{\frac{n-1}{2}}s^{q+\frac{\mu}{2}}\\
&\quad\cdot(t+2-r)^{-q-\frac{\mu}{2}}(t+2-r)^{-1}ds\\
&\lesssim t^{\frac{\mu}{2}-\frac{n-1}{2}}(t+2-r)^{\frac{n-1}{2}-q-\frac{\mu}{2}-1}\int_{0}^{\infty}e^{-s}s^{-\frac{n-1}{2}+q+\frac{\mu}{2}}ds\\
&\lesssim t^{\frac{\mu}{2}-\frac{n-1}{2}}(t+2-r)^{\frac{1}{p}-1}\int_{0}^{\infty}e^{-s}s^{-\frac{1}{p}}ds\\
&\lesssim t^{\frac{\mu}{2}-\frac{n-1}{2}}(t+2-r)^{\frac{1}{p}-1}\bigg(\int_{0}^{1}e^{-s}s^{-\frac{1}{p}}ds+\int_{1}^{\infty}e^{-s}s^{-\frac{1}{p}}ds\bigg)\\
&\lesssim t^{\frac{\mu}{2}-\frac{n-1}{2}}(t+2-r)^{\frac{1}{p}-1}.
\end{aligned}
\end{equation}
Thus, we obtain
\begin{equation}
\begin{aligned}
|\partial_tb_q|\lesssim |I|+|II|\lesssim t^{\frac{\mu}{2}-\frac{n-1}{2}}(t+2-r)^{\frac{1}{p}-1},
\end{aligned}
\end{equation}
which leads to the asymptotic behavior \eqref{asymfortbq}.
\end{proof}

In the following we will establish the lifespan estimate stated in Theorem \ref{hddamped}. With $b_q$ in hand, we consider the functional
\begin{equation}\label{C10}
\begin{aligned}
Y(M)=\int_{1}^{M}\bigg(\int_{\frac{t}{2}}^{t}\int_{\R^n}\tau ^{\alpha}|u|^p(\tau,x)\eta_t^{2p'}(\tau)b_q(\tau,x)dxd\tau \bigg)t^{-1}dt,
\end{aligned}
\end{equation}
where $1\le t\le M$. Next we will establish both the lower and upper bound for $Y(M)$, following the idea in \cite{Ike2, Ike3}. Recalling the proof of (38) in \cite{FLT}, we have the lower bound for the spacetime integral of the nonlinear term
\begin{equation}\label{S9}
\begin{aligned}
\int_{\frac t2}^t\int_{\R^n}|u|^ps^{\alpha}\eta_t^{2p'}(s)dxds \geq C_*C_1^p \e^pt^{n+\alpha-\frac{n-1+\mu}{2}p},
\\
\end{aligned}
\end{equation}
then by combining \eqref{asymforbq} we see that
\begin{equation}\label{C11}
\begin{aligned}
Y'(M)&:= M^{-1}\int^{M}_{\frac{M}{2}}\int_{\R^n}\tau ^{\alpha}|u|^p(\tau,x)\eta_M^{2p'}(\tau)b_q(\tau,x)dxd\tau\\
&\geq C_*C_1^p\e^p  M^{-1}M^{n+\alpha-\frac{n-1+\mu}{2}p-\frac{n-\mu-1}{2}+\frac{1}{p}} \\
&\geq C  \e^pM^{-1}M^{\frac{2+(n+\mu+2\alpha+1)p-(n+\mu-1) p^2}{2p}}\\
&\geq C \e^pM^{-1},
\end{aligned}
\end{equation}
where we used the fact
\begin{equation}\label{C12}
\begin{aligned}
\frac{n-\mu-1}{2}-\frac{1}{p}=n+\alpha-\frac{n+\mu-1}{2}p
\end{aligned}
\end{equation}
for $p=p_S(n+\mu,\alpha)$. Integrating \eqref{C11} from $1$ to $M$ yields
\begin{equation}\label{C13}
\begin{aligned}
Y(M)&\geq C \e^p \ln M,
\end{aligned}
\end{equation}
which means
\begin{equation}\label{C14}
\begin{aligned}
Y(t)\geq C \e^p \ln t,~~~ t\geq 1.
\end{aligned}
\end{equation}

On the other hand, if we exchange the integrating order of $\tau$ and $t$ in \eqref{C10}, and note that $t\le 2\tau$ from the definition of $Y(M)$, we get

\begin{equation}\label{C15}
\begin{aligned}
Y(M)&\leq\int_{\frac{t}{2}}^{M}\int_{\R^n}\tau ^{\alpha}|u|^p(\tau,x)b_q(\tau,x)\int_{1}^{{\min(M, 2\tau)}}\eta_t^{2p'}(\tau)t^{-1}dtdxd\tau  \\
&=\int_{\frac{t}{2}}^{M}\int_{\R^n}\tau^{\alpha}|u|^p(\tau,x)b_q(\tau,x)
\int_{{\max(\frac{\tau}M, \frac12)}}^{\tau}\eta^{2p'}(s)s^{-1}dsdxd\tau  \\
&\leq\int_{\frac{t}{2}}^{M}\int_{\R^n}\tau^{\alpha}|u|^p(\tau,x)b_q(\tau,x)
\int_{{\max(\frac{\tau}M, \frac12)}}^{1}\eta^{2p'}(s)s^{-1}dsdxd\tau  \\
&\leq\int_{\frac{t}{2}}^{M}\int_{\R^n}\tau^{\alpha}|u|^p(\tau,x)b_q(\tau,x)\eta^{2p'}(\frac{\tau}{M})\int_{\frac{\tau}{M}}^{1}s^{-1}dsdxd\tau  \\
&\leq\int_{\frac{t}{2}}^{M}\int_{\R^n}\tau^{\alpha}|u|^p(\tau,x)b_q(\tau,x)
\eta^{2p'}_M(\tau)\int_{{\frac12}}^{1}s^{-1}dsdxd\tau  \\
&\leq \log2\int_{\frac{t}{2}}^{M}\int_{\R^n}\tau^{\alpha}|u|^p
(\tau,x)b_q(\tau,x)\eta^{2p'}_M(\tau)dxd\tau,  \\
\end{aligned}
\end{equation}
where we used the fact that $\eta$ is nonincreasing. The following picture is helpful to understand the exchange of the integrating order in the above inequality
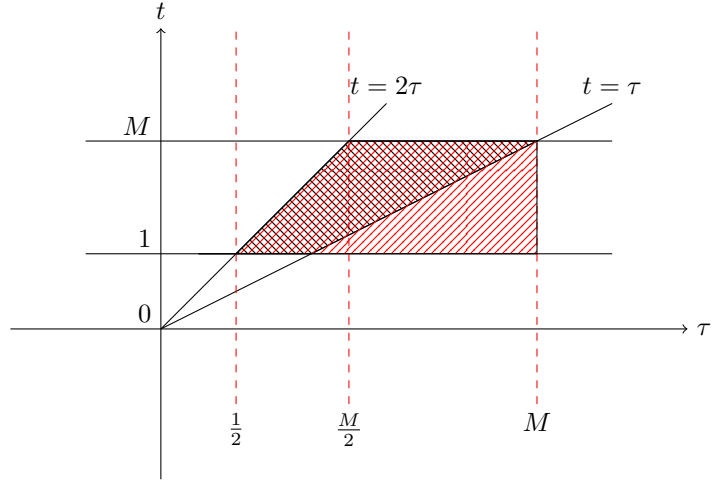
\begin{center}
\begin{tikzpicture}
\draw[-] (0,0.2) node[left] {$0$};
\draw[-] (0,1.2) node[left] {$1$};
\draw[-] (0,2.7) node[left] {$M$};
\draw[very thin] (2,2) node[above] {};
\draw(0,0) -- (3,3)node[above] {$t=2\tau$};
\draw(0,0) -- (6,3)node[above] {$t=\tau$};
\draw[->](-2,0)--(7,0)node[right]{$\tau$};
\draw[->](0,-2)--(0,4)node[above]{$t$};
\draw(-1,1) -- (6,1)node[left]{};
\draw(-1,2.5) -- (6,2.5)node[left] {};
\draw[dashed] [red](1,-1)--(1,4) node[above] {$$};
\draw[-] (1,-1) node[below] {$\frac{1}{2}$};
\draw[dashed][red] (2.5,-1)--(2.5,4) node[above] {$$};
\draw[-] (2.5,-1) node[below] {$\frac{M}{2}$};
\draw[dashed] [red](5,-1)--(5,4) node[above] {$$};
\draw[-] (5,-1) node[below] {$M$};
\draw[pattern color=black!,pattern=north west lines](1,1)--(2.5,2.5)--(5,2.5)--(2,1)--(1/2,1);
\draw[pattern color=red!,pattern=north east lines](1,1)--(2.5,2.5)--(5,2.5)--(5,1)--(1/2,1);
\end{tikzpicture}
\end{center}
By \eqref{dh}, one has
\begin{equation}\label{dh1}
\begin{aligned}
&\lim_{t\rightarrow 0^+}\left(-h'(\lambda t)+\mu\frac{h(\lambda t)}{\lambda t}\right)\\
=&\lim_{t\rightarrow 0^+}(1-\lambda)\mu(\lambda t)^{\frac{\mu-1}{2}}K_{\frac{\mu-1}{2}}(\lambda t)+(\lambda t)^{\frac{\mu+1}{2}}\lambda K_{\frac{\mu+1}{2}}(\lambda t)\\
\thicksim&\lim_{t\rightarrow 0^+}
(1-\lambda)\mu(\lambda t)^{\frac{\mu-1}{2}}\frac12\Gamma\left(\frac{\mu-1}{2}\right)
\left(\frac {\lambda t}{2}\right)^{-\frac{\mu-1}{2}}\\
&+(\lambda t)^{\frac{\mu+1}{2}}\lambda\frac12\Gamma\left(\frac{\mu+1}{2}\right)
\left(\frac {\lambda t}{2}\right)^{-\frac{\mu+1}{2}}\\
\thicksim & 2^{\frac{\mu-3}{2}}(1-\lambda)\Gamma\left(\frac{\mu-1}{2}\right)
+\lambda 2^{\frac{\mu-1}{2}}\Gamma\left(\frac{\mu+1}{2}\right)\triangleq D_0>0,\\
\end{aligned}
\end{equation}
then we have further for $|x|\le 1$
\begin{equation}\label{dh2}
\begin{aligned}
&\lim_{t\rightarrow 0^+}\left(-\partial_tb_q(t,x)+\mu\frac{b_q(t,x)}{t}\right)\\
=&\lim_{t\rightarrow 0^+}\left(-\int_{0}^{1}h'(\lambda t)\phi_\lambda(x)\lambda^{q}d\lambda+\int_{0}^{1}\frac{\mu}{\lambda t}h(\lambda t)\phi_\lambda(x)\lambda^{q}d\lambda\right)\\
=&\lim_{t\rightarrow 0^+}\int_{0}^{1}\left(-h'(\lambda t)+\mu\frac{h(\lambda  t)}{\lambda t}\right)\phi_\lambda(x)\lambda^{q}d\lambda\\
\thicksim &  D_0\int_{0}^{1}\phi_\lambda(x)\lambda^{q}d\lambda\\
\ge &0,
\end{aligned}
\end{equation}
where the recurrence relations \eqref{rela} are used. It follows from \eqref{C15} that
\begin{equation}\label{C5}
\begin{aligned}
Y(t)&\leq \log2\int_{\frac{t}{2}}^{t}\int_{\R^n}\tau^{\alpha}|u|^p(\tau,x)b_q(\tau,x)\eta^{2p'}_t(\tau)dxd\tau\\
&\lesssim \int_{0}^{t}\int_{\R^n}\tau^{\alpha}|u|^p(\tau,x)b_q(\tau,x)\eta^{2p'}_t(\tau)dxd\tau\\
&=\int_0^t\int_{\R^n}\bigg(\partial_\tau^2u-\Delta u +\frac{\mu u_\tau}{\tau}\bigg)b_q(\tau,x)\eta^{2p'}_t(\tau)dxd\tau.
\end{aligned}
\end{equation}
Integrating by parts we have
\begin{equation}\label{C17}
\begin{aligned}
&\int_{\R^n}\lim_{\tau\rightarrow 0^+}\left(-\partial_\tau b_q(\tau,x)+\mu\frac{b_q(\tau,x)}{\tau}\right)u(0,x)dx\\
&+\int_0^t\int_{\R^n}\bigg(\partial_\tau^2u-\Delta u +\frac{\mu u_\tau}{\tau}\bigg)b_q(\tau,x)\eta^{2p'}_M(\tau)dxd\tau\\
=&\int_0^t\int_{\R^n}u\partial_\tau^2\eta_t^{2p'}(\tau)b_qdxd\tau+
2\int_0^t\int_{\R^n}u\partial_\tau\eta_t^{2p'}(\tau)\partial_\tau b_qdxd\tau\\
&-\int_0^t\int_{\R^n}\frac{\mu}{\tau}u\partial_\tau\eta_t^{2p'}(\tau)b_qdxd\tau\\
=:& II_1+II_2+II_3,\\
\end{aligned}
\end{equation}
where by \eqref{dh2} and the assumption of the initial data $u_0(x)$
\begin{equation}
\begin{aligned}\label{C00}
&\int_{\R^n}\lim_{\tau\rightarrow 0^+}\left(-\partial_\tau b_q(\tau,x)+\mu\frac{b_q(\tau,x)}{\tau}\right)u(0,x)dx\\
\thicksim &\int_{|x|\le 1}u_0(x)D_0\int_{0}^{1}\phi_\lambda(x)\lambda^{q}d\lambda dx\\
 \ge& 0.
\end{aligned}
\end{equation}
By using H\"{o}lder inequality, and combining \eqref{asymforbq} and \eqref{asymfortbq}, we may estimate $II_1, II_2, II_3$ as

\begin{equation}\label{C6}
\begin{aligned}
|II_1|\lesssim&t^{-2-\frac{\alpha}{p}}\left(\int_{\frac t2}^t\int_{\R^n}|u|^p\tau^{\alpha}\eta_t^{2p'}b_q
dxd\tau\right)^{\frac1p}\left(\int_{\frac t2}^{t}\int_{\R^n}b_qdxd\tau\right)^{\frac{p-1}p}\\
\lesssim&t^{-2-\frac{\alpha}{p}}\left(\int_{\frac t2}^t\int_{\R^n}|u|^p\tau^{\alpha}\eta_t^{2p'}b_q
dxd\tau\right)^{\frac1p}\left(\int_{\frac t2}^{t}\int_{|x|\le \tau+1}\tau^{-q}(1+r)^{n-1}drd\tau\right)^{\frac{p-1}p}\\
\lesssim&t^{-2-\frac{\alpha}{p}-\frac{q(p-1)}{p}}
\left(\int_{\frac t2}^t\int_{\R^n}|u|^p\tau^{\alpha}\eta_t^{2p'}b_q
dxd\tau\right)^{\frac1p}\left(\int_{\frac t2}^{t}\int_{|x|\le \tau+1}(1+r)^{n-1}drd\tau\right)^{\frac{p-1}p}\\
\lesssim&t^{-2-\frac{\alpha}{p}-n-\alpha+\frac{n+\mu-1}{2}p+\frac{n+\alpha}{p}
-\frac{n+\mu-1}{2}
+\frac{(n+1)(p-1)}{p}}\times\left(\int_{\frac t2}^t\int_{\R^n}|u|^p\tau^{\alpha}\eta_t^{2p'}b_q
dxd\tau\right)^{\frac1p}\\
\lesssim&t^{-\frac{(n+\mu+1+2\alpha)p-(n+\mu-1)p^2+2}{2p}}\left(\int_{\frac t2}^t\int_{\R^n}|u|^p\tau^{\alpha}\eta_t^{2p'}b_q
dxd\tau\right)^{\frac1p}\\
\lesssim&\left(\int_{\frac t2}^t\int_{\R^n}|u|^p\tau^{\alpha}\eta_t^{2p'}b_q
dxd\tau\right)^{\frac1p}.\\
\end{aligned}
\end{equation}
In the same way, we have
\begin{equation}\label{C7}
\begin{aligned}
|II_2|\lesssim&\left(\int_{\frac t2}^t\int_{\R^n}|u|^p\tau^{\alpha}\eta_t^{2p'}b_q
dxd\tau\right)^{\frac1p}\times\left(\int_{\frac t2}^{t}\int_{\R^n}\tau^{-\frac{\alpha}{p}\frac{p}{p-1}}b_q^{-\frac{p'}{p}}(\partial_\tau b_q)^{\frac{p}{p-1}}\tau^{-\frac{p}{p-1}}dxd\tau\right)^{\frac{p-1}p}\\
\lesssim&\left(\int_{\frac t2}^t\int_{\R^n}|u|^p\tau^{\alpha}\eta_t^{2p'}b_q
dxd\tau\right)^{\frac1p}\Bigg(\int_{\frac t2}^{t}\int_{|x|\le \tau+1}\tau^{-\frac{\alpha}{p}\frac{p}{p-1}}\left(\tau^{-q}\right)^{-\frac 1{p-1}}\big(\tau^{\frac{\mu}{2}-\frac{n-1}{2}}\big)^\frac{p}{p-1}\\
&\quad\times\big((\tau+2-r)^{\frac{1}{p}-1}\big)^\frac{p}{p-1}
(1+r)^{n-1}\tau^{-\frac{p}{p-1}}drd\tau\Bigg)^{\frac{p-1}p}\\
\lesssim&\left(\int_{\frac t2}^t\int_{\R^n}|u|^p\tau^{\alpha}\eta_t^{2p'}b_q
dxd\tau\right)^{\frac1p}\left(\int_{\frac t2}^{t}\int_{|x|\le \tau+1}\tau^{-1}(\tau+2-r)^{-1}drd\tau\right)^{\frac{p-1}p}\\
\lesssim&(\ln t)^{\frac{1}{p'}}\left(\int_{\frac t2}^t\int_{\R^n}|u|^p\tau^{\alpha}\eta_t^{2p'}b_q
dxd\tau\right)^{\frac1p},\\
\end{aligned}
\end{equation}
and
\begin{equation}\label{C8}
\begin{aligned}
|II_3|\lesssim&t^{-1}\int_{\frac{t}{2}}^t\int_{\R^n}u\partial_\tau\eta_t^{2p'}(s)b_qdxd\tau\\
\lesssim&t^{-2-\frac{\alpha}{p}}\left(\int_{\frac t2}^t\int_{\R^n}|u|^p\tau^{\alpha}\eta_t^{2p'}b_q
dxd\tau\right)^{\frac1p}\times\left(\int_{\frac t2}^{t}\int_{\R^n}b_qdxd\tau\right)^{\frac{p-1}p}\\
\lesssim&\left(\int_{\frac t2}^t\int_{\R^n}|u|^p\tau^{\alpha}\eta_t^{2p'}b_q
dxd\tau\right)^{\frac1p}.
\end{aligned}
\end{equation}
In conclusion, by combining \eqref{C5}, \eqref{C17}, \eqref{C00}, \eqref{C6}, \eqref{C7}, \eqref{C8} we have
\begin{equation}\label{C27}
\begin{aligned}
Y(t)&\leq II_1+II_2+II_3\\
&\lesssim(\ln t)^{\frac{1}{p'}}\left(\int_{\frac t2}^t\int_{\R^n}|u|^p\tau^{\alpha}\eta_t^{2p'}b_q
dxd\tau\right)^{\frac1p}\\
&+\left(\int_{\frac t2}^t\int_{\R^n}|u|^p\tau^{\alpha}\eta_t^{2p'}b_q
dxd\tau\right)^{\frac1p}\\
&\lesssim(\ln t)^{\frac{1}{p'}}Z(t)^{\frac{1}{p}},~~~for~t\ge 2,
\end{aligned}
\end{equation}
where
\begin{equation}\label{C28}
\begin{aligned}
Z(t):=\int_{\frac t2}^t\int_{\R^n}|u|^p\tau^{\alpha}\eta_t^{2p'}b_q
dxd\tau.
\end{aligned}
\end{equation}
Noting that $Z=tY'(t)$, then it holds
\begin{equation}\label{C29}
\begin{aligned}
tY'(t)\geq (\ln t)^{1-p}Y^p.
\end{aligned}
\end{equation}
In addition, by combining  \eqref{S9} and \eqref{asymforbq}, we have
\begin{equation}\label{C30}
\begin{aligned}
&tY'(t)=Z(t)=\int_{\frac t2}^t\int_{\R^n}|u|^p\tau^{\alpha}\eta_t^{2p'}b_qdxd\tau\\
\gtrsim& \e^pt^{n+\alpha-\frac{n-1+\mu}{2}p-q}\\
=&\e^p.
\end{aligned}
\end{equation}
Integrating \eqref{C30} from 2 to $t>4$ yields
\begin{equation}\label{C31}
\begin{aligned}
Y(t)\geq Y(2)+c\e^p(\ln t-\ln 2)\gtrsim\e^p\ln t, \quad\forall t\in(4,T_\e).
\end{aligned}
\end{equation}
Similarly, integrating \eqref{C29} from $T_1$ to $T_2$ yields
\begin{equation}\label{C32}
\begin{aligned}
Y(T_2)^{1-p}&\leq Y(T_1)^{1-p}-C(p-1)\int_{T_1}^{T_2}\tau^{-1}(\ln \tau)^{1-p}d\tau\\
&\leq (\e^p\ln T_1)^{1-p}-C(p-1)\int_{T_1}^{T_2}(\ln\tau)^{1-p}d\ln\tau\\
&\leq  (\e^p\ln T_1)^{1-p}-C(p-1)\int_{\ln T_1}^{\ln T_2}s^{1-p}ds,\quad \forall 4<T_1<T_2<T_\e.
\end{aligned}
\end{equation}
Let $T_2\rightarrow T_\e$, and noting that $Y(T)\ge 0$, we get from \eqref{C32} that
\begin{equation}\label{C33}
\begin{aligned}
\int_{\ln T_1}^{\ln T_\e}s^{1-p}ds\lesssim (\e^p\ln T_1)^{1-p}.
\end{aligned}
\end{equation}
Set $T_1=\sqrt{T_\e}$ in \eqref{C33}, it is easy to get
\begin{equation}\label{C34}
\begin{aligned}
\ln T_\e\lesssim \e^{-p(p-1)},
\end{aligned}
\end{equation}
which leads to the desired lifespan estimate \eqref{lf2} in Theorem \ref{hddamped}.
\section*{Acknowledgment}
\par
A part of this work was completed when the first author, Mengting Fan, visited Tohoku University.
She would like to express the sincere thank to the third author, Hiroyuki Takamura,
for the warm hospitality and helpful discussion.
\par
The first author is supported by the International Office and School of Mathematical Sciences of Zhejiang Normal University. The second author is partially supported by NSFC (No.12271487, W2521007).
The third author is partially supported by the Grant-in-Aid for Scientific Research(A) (No.22H00097), Japan Society for the Promotion of Science.


\bibliographystyle{plain}

\end{document}